\documentclass[11pt,reqno]{amsart}

\usepackage{graphics,graphicx,fontenc}
\usepackage{epsfig,psfrag,manfnt}
\usepackage{bbm,subfigure,rotating,bm,float,mathdots,wasysym}
\usepackage[all,cmtip]{xy}
\usepackage{amssymb,amsmath,amsfonts,amsthm,color}

\makeatletter
\newcommand*\bigcdot{\mathpalette\bigcdot@{.5}}
\newcommand*\bigcdot@[2]{\mathbin{\vcenter{\hbox{\scalebox{#2}{$\m@th#1\bullet$}}}}}
\makeatother

\newcommand{\calA}{\mathcal{A}}

\newcommand{\calD}{\mathcal{D}}
\newcommand{\calE}{\mathcal{E}}

\newcommand{\calT}{\mathcal{T}}

\newcommand{\mC}{\mathbb{C}}
\newcommand{\mD}{\mathbb{D}}

\newcommand{\mN}{\mathbb{N}}

\newcommand{\mR}{\mathbb{R}}

\newcommand{\mZ}{\mathbb{Z}}

\newcommand{\tsr}{{\textrm{tsr}}}

\newtheorem{theorem}{Theorem}[section]
\newtheorem{lemma}[theorem]{Lemma}

\newtheorem{proposition}[theorem]{Proposition}

\theoremstyle{definition}
\newtheorem{remarks}[theorem]{Remarks}

\theoremstyle{definition}
\newtheorem{definition}[theorem]{Definition}

\theoremstyle{definition}

\theoremstyle{definition}

\begin{document}



\title[Topological stable rank of $\calE'(\mR)$]{Topological stable rank of $\calE'(\mR)$}

\author[A. Sasane]{Amol Sasane}
\address{Department of Mathematics \\London School of Economics\\
    Houghton Street\\ London WC2A 2AE\\ United Kingdom}
\email{A.J.Sasane@lse.ac.uk}

\begin{abstract} 
  The set $\calE'(\mR)$ of all compactly supported distributions, with the
  operations of addition, convolution, multiplication by
  complex scalars, and with the strong dual topology is a topological
  algebra.  In this article, it is shown that the topological stable
  rank of $\calE'(\mR)$ is $2$.
\end{abstract}

\maketitle

\section{Introduction} 

\noindent The aim of this article is to show that the topological stable rank (a notion from 
topological $K$-theory, recalled below) of $\calE'(\mR)$ is $2$, where $\calE'(\mR)$ is the classical topological algebra
of compactly supported distributions, with the strong dual topology $\beta(\calE',\calE)$, 
pointwise vector space operations, and convolution taken as multiplication.

We recall some key notation and facts about $\calE'(\mR)$ in
Section~\ref{section_Eprime} below, including its strong dual topology
$\beta(\calE',\calE)$, and in Section~\ref{section_tsr}, we will
recall the notion of topological stable rank of a topological algebra.

We will prove our main result, stated below, in Sections
\ref{section_geq_2} and \ref{section_leq_2}.

\begin{theorem}
\label{main_thm}$\;$

\noindent 
Let $\calE'(\mR)$ be the algebra of all compactly supported
distributions on $\mR$, with
\begin{itemize}
\item pointwise addition, and pointwise multiplication by complex scalars,
\item convolution taken as the multiplication in the algebra, and
\item the strong dual topology $\beta(\calE',\calE)$.
\end{itemize}
Then the topological stable rank of $\calE'(\mR)$ is equal to $2$.
\end{theorem}

\section{The topological algebra $\calE'(\mR)$}
\label{section_Eprime}

\noindent
For background on topological vector spaces and distributions, we
refer to \cite{Bar}, \cite{Hor}, \cite{Rud}, \cite{Sha}, \cite{Sch} and \cite{Tre}.

Let $\calE(\mR)=C^\infty(\mR)$ be the space of functions
$\varphi:\mR\rightarrow \mC$ that are infinitely many times
differentiable.  We equip $\calE(\mR)$ with the topology of uniform
convergence on compact sets for the function and its derivatives. This is
defined by the following family of seminorms: for a compact
subset $K$ of $\mR$, and $M\in \{0,1,2,3\cdots\}=\mZ_{ \tiny \small\scriptscriptstyle{\geq 0}}$, we define
$$
p_{\scriptscriptstyle{K,M}}(\varphi)=\sup_{0\leq m\leq M}\sup_{x\in K}|\varphi^{(m)}(x)|\;\;\textrm{ for }\varphi\in \calE(\mR).
$$
The space $\calE(\mR)$ is 
\begin{itemize}
 \item metrizable,
 \item a Fr\'echet space, and 
 \item a Montel space;
\end{itemize}
see e.g. \cite[Example 3, p.239]{Hor}.

\smallskip 

\noindent 
By a topological algebra, we mean the following:

\begin{definition}[Topological algebra]$\;$\label{def_top_alg}

\noindent A complex algebra $\calA$ is called a {\em topological algebra} 
if it is equipped with a topology $\calT$ making the following maps 
continuous, with the product topologies on the domains:
\begin{itemize}
\item $\calA\times \calA\owns \;(a,b)\mapsto a+b\;\in \calA$
\item $\mC\times \calA\owns \;(\lambda, a)\mapsto \lambda \cdot a \;\;\in \calA$ 
\item $\calA\times \calA\owns\; (a,b)\mapsto ab\;\!\;\;\;\;\;\in\calA$
\end{itemize}
\end{definition}

\noindent 
We equip the dual space $\calE'(\mR)$ of $\calE(\mR)$ with the strong
dual topology $\beta(\calE',\calE)$, defined by the seminorms
$$
p_{\scriptscriptstyle{B}}(T)=\sup_{\varphi\in B} |\langle T, \varphi\rangle|,
$$
for bounded subsets $B$ of $\calE(\mR)$. Then $\calE'(\mR)$, being the strong dual of the Montel space $\calE(\mR)$, is a
Montel space too \cite[5.9, p.147]{Sha}. This has the consequence that a sequence in $\calE'(\mR)$ is convergent in the 
$\beta(\calE',\calE)$ topology if and only if it is convergent in the weak dual/weak-$\ast$ topology $\sigma(\calE',\calE)$ 
of pointwise convergence on $\calE(\mR)$; see e.g. \cite[Corollary~1, p.358]{Tre}.

As usual, let $\calD(\mR)$ denote the space of all compactly supported functions from 
$C^\infty(\mR)$, and $\calD'(\mR)$ denote the space of all distributions. 
The vector space $\calE'(\mR)$ can be identified with the subspace of
$\calD'(\mR)$ consisting of all distributions having compact
support. If $\calD'(\mR)$ is also equipped with its strong dual
topology, then one has a continuous injection
 $
\calE'(\mR)\hookrightarrow \calD'(\mR). 
$ 
For $T,S\in \calE'(\mR)$, we define their convolution
$T\ast S\in \calE'(\mR)$ by
$$
\langle T\ast S,\varphi\rangle=
\Big\langle T,\;\big[\;x\mapsto  \big\langle S,\varphi (x+\bigcdot)\big\rangle \;\big]\;\Big\rangle,
\quad \varphi \in \calE(\mR).
$$
The map $\ast:\calE'(\mR)\times \calE'(\mR)\rightarrow \calE'(\mR)$ is
(jointly) continuous; see for instance \cite[Chapter VI, \S3, Theorem IV, p. 157]{Sch}.

Thus $\calE'(\mR)$, endowed with the strong dual topology, forms a topological algebra 
with pointwise vector space operations, 
and with convolution taken as multiplication. 
The multiplicative identity element is $\delta_{\scriptscriptstyle{0}}$, the Dirac delta distribution
supported at $0$. In general, we will denote by $\delta_a$ the Dirac
delta distribution supported at $a\in \mR$.

We also recall that the Fourier-Laplace 
transform of a compactly supported distribution
$T\in \calE'(\mR)$ is an entire function, given by
$$
\widehat{T}(z)=\Big\langle T, \;\big(x\mapsto e^{-2\pi i x z}\big)\;\Big\rangle\quad (z\in \mC),
$$
see e.g. \cite[Proposition~29.1, p.307]{Tre}.

\section{Topological stable rank} 
\label{section_tsr}

\noindent
An analogue of the Bass stable rank (useful in algebraic
$K$-theory) for topological rings,  called the topological stable rank, was introduced  in the seminal article 
\cite{Rie}. 

\goodbreak 

\begin{definition}[Unimodular tuple, Topological stable rank]$\;$

\noindent 
  Let $\calA$ be a commutative unital topological algebra with multiplicative identity element denoted by $1$, endowed with
  a topology $\mathcal T$. 
  
  \noindent We define $\calA^n:=\calA\times \cdots \times \calA$ ($n$ times), endowed with the product topology.
\begin{itemize}
\item (Unimodular $n$-tuple).$\;$

  \noindent Let $n\in \mN:=\{1,2,3,\cdots\}$.  We
  call an $n$-tuple $(a_1,\cdots,a_n)\in \calA^n$ {\em unimodular} if
  there exists $(b_1,\cdots,b_n)\in \calA^n$ such that the B\'ezout
  equation
 $
 a_1 b_1+\cdots+a_nb_n=1
 $ 
 is satisfied. The set of all unimodular $n$-tuples is denoted by
 $U_n(\calA)$. Note that $U_1(\calA)$ is the group of invertible
 elements of $\calA$. An element from $U_2(\calA)$ is referred to as a 
 {\em coprime} pair. It can be seen that if $U_n(\calA)$ is dense in $\calA^n$, then 
 $U_{n+1}(\calA)$ is dense in $\calA^{n+1}$. 
\item (Topological stable rank). $\;$

\noindent 
If there exists a least natural number
  $n\in \mN$ for which $U_n(\calA)$ is dense in $\calA^n$, then that $n$ is called the 
  {\em topological stable rank} of $\calA$, denoted by 
  ${\rm tsr}\;\! \calA$. If no such $n$ exists, then 
  $\tsr \calA$ is said to be infinite.
\end{itemize} 
\end{definition}

\noindent While the notion of topological stable rank was introduced in the
context of {\em Banach} algebras, the above extends this notion in a natural manner to 
topological algebras. The topological stable rank of many concrete Banach
algebras  has been determined previously in several works (e.g. \cite{DF}, \cite{Sua1}, \cite{Sua2}). In this article, we determine the topological stable rank of the classical topological algebra $\calE'(\mR)$
from Schwartz's distribution theory.

\section{$\mathrm{tsr}(\calE'(\mR))\geq 2$}
\label{section_geq_2}

\noindent The idea is that if tsr$(\calE'(\mR))$ were 1, then we could approximate any $T$ from $\calE'(\mR)$ 
by compactly supported distributions whose Fourier transform would be zero-free, and by an application of 
Hurwitz Theorem, $\widehat{T}$ would need to be zero-free too, which gives a contradiction, since we can easily choose 
$T$ at the outset to not allow this. 

\begin{proposition}
\label{prop_geq}
$\textrm{\em tsr}(\calE'(\mR))\geq 2$.
\end{proposition}
\begin{proof}
Suppose on the contrary that $\tsr(\calE'(\mR))=1$. Let 
$$
T=\frac{\delta_{\scriptscriptstyle{-1}}-\delta_{\scriptscriptstyle{1}}}{2i}\in \calE'(\mR).
$$
By our assumption, $U_1(\calE'(\mR))$ is dense in $(\calE'(\mR),\beta(\calE',\calE))$.  But
then the set $U_1(\calE'(\mR))$ is also sequentially dense: This is a consequence of 
the fact that a subset $F$ of $\calE'(\mR)$ is closed in 
$\beta(\calE',\calE)$ if and only if it is sequentially closed.  (See
\cite[Satz 3.5, p.231]{Obe}, which says that $E'$, with the $\beta (E',E)$-topology, 
is sequential whenever $E$ is Fr\'echet-Montel. A locally convex space $F$ is
{\em sequential} if any subset of $F$ is closed if and only if it is
sequentially closed. If $F$ has this property, then the closure of any
subset equals its sequential closure, and therefore being dense is the
same as being sequentially dense. In our case,
$E=\calE(\mR)$ is Fr\'echet-Montel, and so $\calE'(\mR)$ is
sequential. In fact in the remark following \cite[Satz 3.5]{Obe}, the
case of $\calE'(\mR)$ is mentioned as a corollary.)  

Thus there exists
a sequence $(T_n)_{n\in \mN}$ in $U_1(\calE'(\mR))$ such that
$T_n\stackrel{n\rightarrow\infty}{\longrightarrow} T$ in
$\calE'(\mR)$. But since each $T_n$ is invertible in $\calE'(\mR)$,
there exists a sequence $(S_n)_{n\in\mN}$ in $ \calE'(\mR)$ such that
$$
T_n\ast S_n=\delta_{\scriptscriptstyle{0}}\;\; \textrm{ for all } n\in \mN.
$$
Taking the Fourier-Laplace transform, we obtain
$$
\widehat{T_n}(z) \cdot \widehat{S_n}(z)=1 \;\;\textrm{ for all }z\in
\mC\textrm{ and all } n\in \mN.
$$
In particular, the entire functions $\widehat{T_n}$ are all zero-free. 

But as $T_n\stackrel{n\rightarrow\infty}{\longrightarrow} T$ in $\calE'(\mR)$, we now
show that $(\widehat{T_n})_{n\in \mN}$ converges to $\widehat{T}$
uniformly on compact subsets of $\mC$ as $n\rightarrow \infty$.  The
pointwise convergence of $(\widehat{T_n})_{n\in \mN}$ to $\widehat{T}$
is clear by taking the test function $x\mapsto e^{-2\pi i x z}$:
$$
\widehat{T_n}(z)=\langle T_n,e^{-2\pi iz\;\!\bigcdot }\rangle 
\stackrel{n\rightarrow\infty}{\longrightarrow} 
\langle T,e^{-2\pi iz\;\!\bigcdot }\rangle=\widehat{T}(z).
$$
Now for any $\varphi \in \calE(\mR)$, we know that the sequence
$(\langle T_n,\varphi\rangle)_{n\in \mN}$ converges to
$\langle T,\varphi\rangle$, and in particular, the set
$$
\Gamma(\varphi):=\{\langle T_n,\varphi\rangle:n\in \mN\}
$$
is bounded, for every $\varphi\in \calE(\mR)$. By the
Banach-Steinhaus Theorem for Fr\'echet spaces (see for example
\cite[Theorem 2.6, p.45]{Rud}), applied in our case to the Fr\'echet
space $\calE(\mR)$, we conclude that 
$$
\Gamma=\{T_n:n\in \mN\}
$$ 
is
equicontinuous. Thus for every $\epsilon>0$, there exists a
neighbourhood $V$ of $0$ in $\calE(\mR)$ such that
$T_n(V)\subset B(0,\epsilon):=\{z\in \mC:|z|<\epsilon\}$ for all
$n\in \mN$. From here it follows that there exist $M\in \mZ_{\scriptscriptstyle{\geq 0}}$,
$R>0$ and $C>0$ such that
$$
|\langle T_n,\varphi\rangle|\leq C\Big(1+\sup_{0\leq m\leq M} \sup_{|x|\leq R} |\varphi^{(m)}(x)|\Big) .
$$
By taking $\varphi=(x\mapsto e^{-2\pi i x z})$ in the above, we obtain 
$$
|\widehat{T_n}(z)|\leq C'(1+|z|)^{M'} e^{R' |z|},\quad z\in \mC, \;n\in \mN.
$$
Also, by the Payley-Wiener-Schwartz Theorem \cite[Theorem~4.12, p.139]{Bar} for $T\in \calE'(\mR)$, 
we have 
$$
|\widehat{T}(z)|\leq C''(1+|z|)^{M''} e^{R'' |z|},\quad z\in \mC, \;n\in \mN.
$$
It now follows that for some constants $C_*,M_*,R_*$ that  
$$
|\widehat{T_n}(z)-\widehat{T}(z)|\leq C_*(1+|z|)^{M_*} e^{R_* |z|},\quad z\in \mC, \;n\in \mN.
$$
But this means that the pointwise convergent sequence
$(\widehat{T_n})_{n\in \mN}$ of entire functions is uniformly bounded
on compact subsets of $\mC$ (that is, the sequence constitutes a
normal family).  Then it follows from Montel's Theorem (see e.g. 
\cite[Exercise~9.4, p.157]{Ull}) that $(\widehat{T_n})_{n\in \mN}$
converges to $\widehat{T}$ uniformly on compact subsets of $\mC$ as
$n\rightarrow \infty$.

But now by Hurwitz Theorem (see e.g. \cite[Exercise~5.6, p.85]{Ull}), and considering, say, the compact set
$K=\{z\in \mC:|z|\leq 1\}$, we conclude that $\widehat{T}$ must be
either be identically zero on $K$ or that it must be zero-free in $K$.
But $\widehat{T}$ is neither:
$$
\widehat{T}(z)=\frac{e^{2\pi i z}-e^{-2\pi i z}}{2i}=\sin (2\pi z),
$$
a contradiction. Hence $\tsr (\calE'(\mR))\geq 2$.
\end{proof}

\noindent An alternative proof of Proposition~\ref{prop_geq}, suggested by Peter Wagner \cite{Wag}, is as follows. 
 The theorem of supports (\cite[Theorem~4.3.3]{Hor})
implies that $U_1(\calE'(\mR))$ equals the set of nonzero multiples
of $\delta_a$ for arbitrary $a\in \mR$, and this set is not dense in $\calE'(\mR)$. We give the 
details below. First, one can show the following structure result for $U_1(\calE'(\mR))$.

\begin{proposition}
\label{prop_PW}
$U_1(\calE'(\mR))=\{c\delta_a: a\in \mR, \;0\neq c\in \mC\}$. 
\end{proposition}
\begin{proof} 
It is clear that $\{c\delta_a:a\in \mR,\;0\neq c\in \mC \}\subset U_1(\calE'(\mR))$ since 
$$
(c\delta_a)\ast (c^{-1} \delta_{-a})=\delta_0.
$$
Now suppose that $T\in U_1(\calE'(\mR))$. Then there exists an $S\in \calE'(\mR)$ such that 
$T\ast S=\delta_0$. By the Theorem on Supports \cite[Theorem~4.3.3, p.107]{Hor}, we have 
 $$
\textrm{c.h.supp}(T\ast S)=\textrm{c.h.supp}(T)+\textrm{c.h.supp}(S),
$$ 
where, for a distribution  $E\in \calE'(\mR)$, the notation $\textrm{c.h.supp}(E)$ is used for the closed convex hull of 
$\textrm{supp}(E)$, that is, the intersection of all closed convex sets containing $\textrm{supp}(E)$. 
So we obtain 
$$
\{0\}=\textrm{c.h.supp}(\delta_0)= \textrm{c.h.supp}(T\ast S)=\textrm{c.h.supp}(T)+\textrm{c.h.supp}(S),
$$
from which it follows that $\textrm{c.h.supp}(T)=\{a\}$ and  
$\textrm{c.h.supp}(S)=\{-a\}$ for some $a\in \mR$. But then 
 also  $\textrm{supp}(T)=\{a\}$ and  
$\textrm{supp}(S)=\{-a\}$. As distributions with 
support in a point $p$ are linear combinations of the Dirac delta distribution  $\delta_p$ 
and its derivatives $\delta_p^{(n)}$ \cite[Theorem~24.6, p.266]{Tre}, we conclude that 
$S$ and $T$ have the form 
\begin{eqnarray*}
 T&=&\sum_{n=0}^N t_n \delta_a^{(n)},\\
 S&=&\sum_{m=0}^M s_m \delta_{-a}^{(m)}, 
\end{eqnarray*}
for some integers $N,M\geq 0$ and some complex numbers $t_n,s_m$ ($0\leq n\leq N$, $0\leq m\leq M$). 
Now $T\ast S=\delta_0$ implies that $N=M=0$ and $t_0s_0=1$, thanks to the 
 the linear independence of the set 
$$
\{\delta_{-a},\delta_{-a}',\delta_{-a}'',\cdots\}\cup\{\delta_{0},\delta_{0}',\delta_{0}'',\cdots\}\cup 
\{\delta_{a},\delta_{a}',\delta_{a}'',\cdots\}
$$
in the complex vector space $\calE'(\mR)$.  
In particular $t_0\neq 0$. 
Thus 
$$
T= t_0 \delta_a\in \{c\delta_p: p\in \mR, \;0\neq c\in \mC  \}.
$$ 
Consequently, $U_1(\calE'(\mR))=\{c\delta_a: a\in \mR, \;0\neq c\in \mC  \}$.
\end{proof}

\noindent Based on the above, we can now give the following 
alternative proof of Proposition~\ref{prop_geq}.

\begin{proof} We show $U_1(\calE'(\mR))$ is not dense in $(\calE'(\mR),\beta(\calE',\calE))$. If it were, then it would 
be sequentially dense too, and so for each element $T$ of $\calE'(\mR)$, there would exist a sequence in $U_1(\calE'(\mR))$  that converges to $T$ in the $\beta(\calE',\calE)$ topology, and hence also in the 
$\sigma(\calE',\calE)$ topology. But we now show that 
 $\delta_0'\in \calE'(\mR)$ cannot be approximated in the $\sigma(\calE',\calE)$ topology by elements from  $U_1(\calE'(\mR))=\{c\delta_a: a\in \mR, \;0\neq c\in \mC  \}$. Suppose, on the contrary, that $(c_n\delta_{a_n})_{n\in \mN}$ 
converges to $\delta_0'$ in $(\calE'(\mR),\sigma(\calE',\calE))$. 

We first note that $(a_n)_{n\in \mN}$ is bounded. For if not, then there exists a subsequence 
$(a_{n_k})_{k\in \mN}$ of $(a_n)_{n\in \mN}$ such that $|a_{n_k}|>2$ for all $k\in \mN$. Now choose a $\varphi\in \calD(\mR)$ 
such that $\varphi'(0)=1$ and $\varphi\equiv 0$ on $\mR\setminus(-1,1)$. Then we arrive at the contradiction that 
$$
0=c_{n_k}\cdot 0= c_{n_k}\cdot \varphi(a_{n_k})= \langle c_{n_k}\delta_{a_{n_k}},\varphi\rangle  \stackrel{k\rightarrow\infty}{\longrightarrow} \langle \delta_0',\varphi\rangle=
-\varphi'(0)=-1.
$$
So $(a_n)_{n\in \mN}$ is bounded. 

Next we show that $(c_n)_{n\in \mN}$ converges to $0$. Let $R>0$ be such that $|a_n|<R$ for all $n\in \mN$. 
Let $\psi\in \calD(\mR)$ be such that $\psi\equiv 1$ on $[-R,R]$. Then we have 
$$
c_n=c_n\cdot 1= c_n\cdot \psi(a_n)=\langle c_n \delta_{a_n},\psi \rangle
\stackrel{n\rightarrow\infty}{\longrightarrow} \langle \delta'_0,\psi\rangle=-\psi'(0)=0.
$$
Finally, we show that $(c_n\delta_{a_n})_{n\in \mN}$ 
converges to $0$ in $(\calE'(\mR),\sigma(\calE',\calE))$. 
For any $\chi \in \calD(\mR)$, we have 
$$
|\langle c_n\delta_{a_n},\chi\rangle|=|c_n|\cdot |\chi(a_n)|\leq |c_n|\cdot \|\chi \|_\infty\stackrel{n\rightarrow \infty}{\longrightarrow }
0\cdot \|\chi \|_\infty=0.
$$
So $(c_n\delta_{a_n})_{n\in \mN}$ 
converges to $0$ in $(\calE'(\mR),\sigma(\calE',\calE))$. But this is a contradiction, since $0\neq \delta'_0$ in $\calE'(\mR)$. Consequently, $U_1(\calE'(\mR))$ is not dense in $(\calE'(\mR),\beta(\calE',\calE))$, 
and so $\tsr(\calE'(\mR))\geq 2$.
\end{proof}

\section{$\mathrm{tsr}(\calE'(\mR))\leq 2$}
\label{section_leq_2}

\noindent The idea is to reduce the  determination of tsr$(\calE'(\mR))$ to tsr($\mC[z]$) 
of the polynomial ring $\mC[z]$ as follows. Given a pair from $\calE'(\mR)$, we use mollification to make 
a pair in $\calD(\mR)$, and then approximate 
the resulting smooth functions by 
a linear combination of Dirac distributions with uniform spacing. 
The uniform spacing affords 
the identification of the linear combination of Dirac deltas with the ring of polynomials. 

\noindent 
For $n\in \mN$, we define the collection ${\mathbf D}_{n}$ of all `finitely
supported Dirac delta combs' with spacing $1/n$ by
$$
{\mathbf D}_{n}:=\textrm{span}\;\! \{ \delta_{\scriptscriptstyle{k/n}}:k\in \mZ\},
$$
where `span' means the set of  all (finite) linear combinations. 

\begin{lemma}[Approximating a pair of Dirac combs by a {\em unimodular} pair]
\label{lemma_1}
Let $n\in \mN$ and $T,S\in {\mathbf D}_{n}$. Then there exist sequences $(T_k)_{k\in \mN}$ and
$(S_k)_{k\in \mN}$ in ${\mathbf D}_n$, which converge to $T,S$, respectively, in
$(\calE'(\mR),\sigma(\calE',\calE))$, and hence\footnote{Because $\calE'(\mR)$ is a Montel space; see \cite[Corollary~1, p.358]{Tre}.} also in $(\calE'(\mR),\beta(\calE',\calE))$,  and are such that for each $k$,
$(T_k,S_k)\in U_2(\calE'(\mR))$.
\end{lemma}
\begin{proof}
Write 
$\displaystyle 
T=\sum_{\ell=-L}^{L}t_{\scriptscriptstyle{\ell}} \delta_{\scriptscriptstyle{\ell/n}}$, and 
 $\displaystyle  
S=\sum_{\ell=-L}^{L}s_{\scriptscriptstyle{\ell}} \delta_{\scriptscriptstyle{\ell/n}},
$ for some $L\in \mN$, $t_{\scriptscriptstyle{\ell}},s_{\scriptscriptstyle{\ell}} \in \mC$.

\noindent 
Define 
\begin{eqnarray*}
 p_{\scriptscriptstyle{T}}&:=&t_{\scriptscriptstyle{-L}}+t_{\scriptscriptstyle{-L+1}} z+\cdots +t_{\scriptscriptstyle{L}} z^{2L},\\
 p_{\scriptscriptstyle{S}}&:=&s_{\scriptscriptstyle{-L}}+s_{\scriptscriptstyle{-L+1}} z+\cdots +s_{\scriptscriptstyle{L}} z^{2L}.
\end{eqnarray*}
For a given $k\in \mN$, let $\epsilon=1/(2^k \cdot 2L)>0$. Then we can perturb the coefficients of the
polynomials $p_{\scriptscriptstyle{T}}, p_{\scriptscriptstyle{S}}$ within a distance of $\epsilon$ to make them
have no common zeros, that is after perturbation of coefficients they
are coprime in the ring $\mC[z]$. Indeed any polynomial $p_{\scriptscriptstyle{T}}, p_{\scriptscriptstyle{S}}$
can be factorized as
$$
p_{\scriptscriptstyle{T}}=C \prod (z-\alpha_{\scriptscriptstyle{\ell}}), \quad 
p_{\scriptscriptstyle{S}}=C'\prod (z-\beta_{\scriptscriptstyle{\ell}}),
$$
and if there is some common zero $\alpha_{\scriptscriptstyle{\ell}}=\beta_{\scriptscriptstyle{\ell'}}$, we simply
replace $\beta_{\scriptscriptstyle{\ell'}}$ by $\beta_{\scriptscriptstyle{\ell'}}+\epsilon'$ with an $\epsilon'$
small enough so that the final coefficients (of this new perturbed polynomial obtained from $p_{\scriptscriptstyle{S}}$), which are polynomial
functions of the zeros, lie within the desired
$\epsilon$ distance of the coefficients of $p_{\scriptscriptstyle{S}}$.  
So we can choose
$\widetilde{t}_{\scriptscriptstyle{-L,k}},\cdots ,\widetilde{t}_{\scriptscriptstyle{L,k}}$ and
$\widetilde{s}_{\scriptscriptstyle{-L,k}},\cdots ,\widetilde{s}_{\scriptscriptstyle{L,k}}$ such that for all 
$\ell=-L,\cdots, L$, we have
$$
|t_{\scriptscriptstyle{\ell}}-\widetilde{t}_{\scriptscriptstyle{\ell,k}}|
<
\frac{1}{2^k}\cdot \frac{1}{2L} \quad \textrm{ and } \quad 
|s_{\scriptscriptstyle{\ell}}-\widetilde{s}_{\scriptscriptstyle{\ell,k}}|
<
\frac{1}{2^k}\cdot \frac{1}{2L} , 
$$
and so that 
\begin{eqnarray*}
 \widetilde{p}_{\scriptscriptstyle{T,k}}
 &:=&
 \widetilde{t}_{\scriptscriptstyle{-L,k}}+\widetilde{t}_{\scriptscriptstyle{-L+1,k}} z+
 \cdots +\widetilde{t}_{\scriptscriptstyle{L,k}} z^{2L},\\
 \widetilde{p}_{\scriptscriptstyle{S,k}}
 &:=&
 \widetilde{s}_{\scriptscriptstyle{-L,k}}+\widetilde{s}_{\scriptscriptstyle{-L+1,k}} z+
 \cdots +\widetilde{s}_{\scriptscriptstyle{L,k}} z^{2L}
\end{eqnarray*}
have no common zeros. Thus $\widetilde{p}_{\scriptscriptstyle{T,k}},\widetilde{p}_{\scriptscriptstyle{S,k}}$
are coprime in $\mC[z]$, and hence there exist polynomials
$q_{\scriptscriptstyle{T,k}},q_{\scriptscriptstyle{S,k}}\in \mC[z]$ (\cite[Corollary~8.5, p.374]{Art}) such
that
$$
\widetilde{p}_{\scriptscriptstyle{T,k}}\cdot q_{\scriptscriptstyle{T,k}}
+\widetilde{p}_{\scriptscriptstyle{S,k}} \cdot q_{\scriptscriptstyle{S,k}}=1.
$$
Set $Q_{\scriptscriptstyle{T,k}}:=z^L q_{\scriptscriptstyle{T,k}}$ and 
$Q_{\scriptscriptstyle{S,k}}:=z^L q_{\scriptscriptstyle{S,k}}$, and 
\begin{eqnarray*}
 P_{\scriptscriptstyle{T,k}}
 &:=&
 \widetilde{t}_{\scriptscriptstyle{-L,k}}z^{-L}+\widetilde{t}_{\scriptscriptstyle{-L+1,k}} z^{-L+1}
 +\cdots +\widetilde{t}_{\scriptscriptstyle{L,k}} z^{L},\\
 P_{\scriptscriptstyle{S,k}}
 &:=&
 \widetilde{s}_{\scriptscriptstyle{-L,k}}z^{-L}+\widetilde{s}_{\scriptscriptstyle{-L+1,k}} z^{-L+1}
 +\cdots +\widetilde{s}_{\scriptscriptstyle{L,k}} z^{L}.
\end{eqnarray*}
Then in the ring $\mC[z,z^{-1}]$ of linear combinations of monomials $z^n$, where $n\in \mZ$ (i.e. the Laurent polynomial ring $\mC[z,z^{-1}]
=\mC[z,w]/\langle zw-1 \rangle$; see for example \cite[p.367]{Art}), we have
\begin{equation}
 \label{Bezout_27_10_2018_11:41}
 P_{\scriptscriptstyle{T,k}}\cdot  Q_{\scriptscriptstyle{T,k}}+ P_{\scriptscriptstyle{S,k}}\cdot Q_{\scriptscriptstyle{S,k}}=1.
\end{equation}
Suppose that $Q_{\scriptscriptstyle{T,k}}$ and $Q_{\scriptscriptstyle{S,k}}$ have the expansions
\begin{eqnarray*}
 Q_{\scriptscriptstyle{T,k}}
 &=&
 \tau_{\scriptscriptstyle{L',k}}z^{L+L'}+\tau_{\scriptscriptstyle{L'-1,k}} z^{L+L'-1}+\cdots 
 +\tau_{\scriptscriptstyle{0,k}} z^{L},\\
 Q_{\scriptscriptstyle{S,k}}
 &=&
 \sigma_{\scriptscriptstyle{L',k}}z^{L+L'}+\sigma_{\scriptscriptstyle{L'-1,k}} z^{L+L'-1}+\cdots 
 +\sigma_{\scriptscriptstyle{0,k}} z^{L}.
\end{eqnarray*}
Finally, set 
$$
T_k:=\sum_{\ell=-L}^{L}\widetilde{t}_{\scriptscriptstyle{\ell,k}} \delta_{\scriptscriptstyle{\ell/n}}, \quad 
S_k:=\sum_{\ell=-L}^{L}\widetilde{s}_{\scriptscriptstyle{\ell,k}} \delta_{\scriptscriptstyle{\ell/n}},
$$
and 
\begin{eqnarray*}
 U_{k}
 &:=&
 \tau_{\scriptscriptstyle{L',k}}\delta_{\scriptscriptstyle{(L+L')/n}}+\tau_{\scriptscriptstyle{L'-1,k}} 
 \delta_{\scriptscriptstyle{(L+L'-1)/n}}+\cdots +\tau_{\scriptscriptstyle{0,k}} \delta_{\scriptscriptstyle{L/n}},\\
 V_{k}
 &:=&
 \sigma_{\scriptscriptstyle{L',k}}\delta_{\scriptscriptstyle{(L+L')/n}}+\sigma_{\scriptscriptstyle{L'-1,k}} 
 \delta_{\scriptscriptstyle{(L+L'-1)/n}}+\cdots +\sigma_{\scriptscriptstyle{0,k}} \delta_{\scriptscriptstyle{L/n}}.
\end{eqnarray*}
Then it follows from \eqref{Bezout_27_10_2018_11:41} that 
\begin{equation}
 \label{Bezout_5_10_2018_12:46}
T_k\ast U_k +S_k \ast V_k=\delta_{\scriptscriptstyle{0}}.
\end{equation}
To see this, we note that $\Phi:\mC[z,z^{-1}]\rightarrow {\mathbf{D}}_n$ given by 
\begin{eqnarray*}
\Phi(z)=\delta_{\scriptscriptstyle 1/n}&\textrm{ and }&\Phi(1)=\delta_{\scriptscriptstyle 0}
\end{eqnarray*}
defines a ring homomorphism, and then \eqref{Bezout_5_10_2018_12:46} above follows by 
applying $\Phi$ on both sides of \eqref{Bezout_27_10_2018_11:41}.
Hence $(T_k,U_k)\in U_2(\calE'(\mR))$. Also, for any $\varphi \in \calE(\mR)$, we have 
\begin{eqnarray*}
\Big|\big\langle (T-T_k),\varphi\big\rangle\Big|
&=&
\left|\sum_{\ell=-L}^{L}(t_{\scriptscriptstyle{\ell}}-\widetilde{t}_{\scriptscriptstyle{\ell,k}}) 
\langle \delta_{\scriptscriptstyle{\ell/n}},\varphi\rangle\right|
\\
&=&\frac{1}{2^k} \cdot \frac{1}{2L}\cdot  2L\cdot  \sup_{x\in [-
\frac{L}{n},\frac{L}{n}]}|\varphi(x)|\\
&=&
\frac{1}{2^k}  \cdot \sup_{x\in [-
\frac{L}{n},\frac{L}{n}]}|\varphi(x)|
\stackrel{k\rightarrow \infty}{\longrightarrow }0.
\end{eqnarray*}
Hence $T_k\stackrel{k\rightarrow \infty}{\longrightarrow }T$ in
$(\calE'(\mR),\sigma(\calE',\calE))$ as $k\rightarrow \infty$. But then this convergence also is 
valid in $(\calE'(\mR),\beta(\calE',\calE))$,  by \cite[Corollary~1, p.358]{Tre}, since $\calE'(\mR)$ is a Montel space.  
Similarly, $S_k\stackrel{k\rightarrow \infty}{\longrightarrow }S$ in
$(\calE'(\mR),\sigma(\calE',\calE))$ as $k\rightarrow \infty$, and again, the convergence holds in $(\calE'(\mR),\beta(\calE',\calE))$.  This completes the proof.
\end{proof}

\begin{lemma}[Approximation in $\calE'(\mR)$ by Dirac combs]
\label{lemma_2}$\;$

\noindent 
Let $T\in \calE'(\mR)$. Then there exists a sequence
$(T_n)_{n\in \mN}$ such that
\begin{itemize}
\item for all $n\in \mN$, $T_n\in {\mathbf D}_n$, and
\item $T_n\stackrel{n\rightarrow \infty}{\longrightarrow} T$ in
  $(\calE'(\mR),\sigma(\calE',\calE))$, and hence also in $(\calE'(\mR),\beta(\calE',\calE))$.
 \end{itemize}
\end{lemma}
\begin{proof} Let $k\in \mN$ be such that the support of $T$ is
  contained in $(-k,k)$. We first produce a mollified
  approximating sequence for $T$. Let $\varphi:\mR\rightarrow [0,\infty)$ be any  test
  function in $\calD(\mR)$ with support in $[-a,a]$ for some
  $a>0$, and such that 
  $$
  \int_\mR \varphi(x) dx =1.
  $$
  Then we know that if we define
 $
\varphi_m(x):=m\cdot \varphi(mx)$ ($m\in \mN$),  
then for each $m$, 
$$
f_m:=T\ast \varphi_m
$$ 
is a smooth function having
a compact support, and moreover,
$$
T\ast \varphi_m\stackrel{m\rightarrow \infty}{\longrightarrow} T
$$ 
in $(\calE'(\mR),\beta(\calE',\calE))$; see for example \cite[Theorem 3.3,
p.97]{Bar}. So the convergence is also valid in $(\calE'(\mR), \sigma(\calE',\calE))$. Moreover, as the support of $ f_m= T\ast \varphi_m$ is
contained in the sum of the supports of $\varphi_m$ and of $T$, 
for all $m$ large enough, say $m\geq M$, we have 
\begin{eqnarray*}
\textrm{supp}(T\ast \varphi_m)
&\subset &\textrm{supp}(T)+\textrm{supp}(\varphi_m)\\
&\subset &\textrm{supp}(T)+[-a/m,a/m]\\
&\subset& [-k,k].
\end{eqnarray*}
From now on, we will assume that $m\geq M$, so that
$\textrm{supp}(f_m)\subset [-k,k]$.  Now we will approximate $f_m$ by
Dirac comb elements. To this end, we define
$$
T_{m, n}:= \sum_{\ell=0}^{n-1} \frac{2k}{n} \cdot f_m\Big(-k+\frac{2k}{n} \ell\Big) \cdot 
\delta_{\scriptscriptstyle{-k+\frac{2k}{n}\ell}}\;\; \in {\mathbf D}_n.
$$
We will show that
$T_{m,n}\stackrel{n\rightarrow\infty}{\longrightarrow} f_m$ in
$(\calE'(\mR), \sigma(\calE',\calE))$.  Let $\psi\in \calE(\mR)$. Then
\begin{eqnarray*}
 \langle T_{m,n},\psi\rangle 
 &=& 
 \left\langle \sum_{\ell=0}^{n-1} \frac{2k}{n} \cdot f_m\Big(-k+\frac{2k}{n} \ell\Big) \cdot
\delta_{\scriptscriptstyle{-k+\frac{2k}{n}\ell}},\; \psi\right\rangle \\
&=& \sum_{\ell=0}^{n-1} \frac{2k}{n}\cdot  f_m\Big(-k+\frac{2k}{n} \ell\Big) 
\left\langle\delta_{\scriptscriptstyle{-k+\frac{2k}{n}\ell}},\;\psi\right\rangle\\
&=& 
\sum_{\ell=0}^{n-1} \frac{2k}{n} \cdot f_m\Big(-k+\frac{2k}{n} \ell\Big)\cdot 
\psi\Big( -k+\frac{2k}{n}\ell\Big).
\end{eqnarray*}
Thus $ \langle T_{m,n},\psi\rangle $ gives a Riemann sum for the
integral of the continuous function $f_m\psi$ with compact support
contained in $[-k,k]$, giving
\begin{eqnarray*}
&& \big|\langle T_{m,n},\psi\rangle -\langle f_{m},\psi\rangle\big|
\\
&=& \!\!\!
 \left|
\sum_{\ell=0}^{n-1} \frac{2k}{n} \!\cdot \!f_m\Big(\!-\!k\!+\!\frac{2k}{n} \ell\Big)\!\cdot \!
\psi\Big( \!-\!k\!+\!\frac{2k}{n}\ell\Big)
\!-\!
\int_{-k}^k \!f_m(x) \psi (x) dx\right|\stackrel{n\rightarrow\infty}{\longrightarrow} 0. 
\end{eqnarray*}
Hence $T_{m,n}\stackrel{n\rightarrow\infty}{\longrightarrow} f_m$ in
$(\calE'(\mR),\sigma(\calE',\calE))$. As $\calE'(\mR)$ is a Montel space, this convergence is also valid in $(\calE'(\mR),\beta(\calE',\calE))$, and the proof is completed.
\end{proof}

\begin{proposition}
\label{prop_leq}
 $\textrm{\em tsr}(\calE'(\mR))\leq 2$.
\end{proposition}
\begin{proof}
  Let $T,S\in \calE'(\mR)$. Throughout this proof, $\calE'(\mR)$ is endowed with the strong dual topology $\beta(\calE',\calE)$, and then $(\calE'(\mR))^2=\calE'(\mR)\times \calE'(\mR)$ is equipped the product topology.
  Let $V$ be a neighbourhood of $(T,S)$ in 
  $(\calE'(\mR))^2$. By Lemma \ref{lemma_2}, it follows that 
  $
  \bigcup\limits_{n\in \mN} ({\mathbf D}_n\times {\mathbf D}_n)
  $ 
  is sequentially dense, and hence dense, in $(\calE'(\mR))^2$. 
  Thus there exists a pair 
  $(T_*,S_*)\in V\cap ({\mathbf D}_n\times {\mathbf D}_n)$ for some $n\in \mN$.  By Lemma~\ref{lemma_1}, there exists a 
  sequence $(T_k,S_k)_{k\in \mN}$ in 
  $({\mathbf D}_n\times {\mathbf D}_n)\cap U_2(\calE'(\mR))$ that converges to $(T_*,S_*)$ in $(\calE'(\mR))^2$. 
  Since $V$ is also a neighbourhood of 
  $(T_*,S_*)$ in $(\calE'(\mR))^2$,  there exists an index $K$ large enough so that 
  for all $k>K$, $(T_k,S_k)\in V$. 
  Consequently, $U_2(\calE'(\mR))$ is 
  dense in $(\calE'(\mR))^2$. 
\end{proof}

\begin{proof}[Proof of Theorem~\ref{main_thm}]
It follows from Propositions~\ref{prop_geq} and \ref{prop_leq} that 
 the topological stable rank of $(\calE'(\mR), +,\cdot,\ast, \beta(\calE',\calE))$ is equal to $2$. 
\end{proof}

\begin{remarks}$\;$

\medskip 

\noindent (1)  From the proofs, it is clear that we have shown that 
$U_1(\calE'(\mR))$ is not dense in $(\calE'(\mR),\sigma(\calE',\calE))$, while 
$U_2(\calE'(\mR))$ is sequentially dense, and hence dense, in 
$(\calE'(\mR))^2$ endowed with the product topology with $\calE'(\mR)$ bearing the $\sigma(\calE',\calE)$ topology. 

However, we note that $\ast:\calE'(\mR)\times \calE'(\mR)\rightarrow \calE'(\mR)$ is not continuous if we use the $\sigma(\calE',\calE)$ topology on $\calE'(\mR)$: For example, in $(\calE'(\mR),\sigma(\calE',\calE))$, we have that 
$
\delta_{\pm n}\stackrel{n\rightarrow\infty }{\longrightarrow} 0, 
$ 
so that in the product topology on $(\calE'(\mR))^2$, we have 
 $
(\delta_n,\delta_{-n})\stackrel{n\rightarrow\infty }{\longrightarrow} (0,0)
$. But $\delta_n\ast \delta_{-n}=\delta_{n-n}=
\delta_0\stackrel{n\rightarrow\infty }{\longrightarrow} \delta_0\neq 0=0\ast 0$. So $(\calE'(\mR), +,\cdot,\ast,\sigma(\calE',\calE))$ is 
not a topological algebra in the sense of our Definition~\ref{def_top_alg}. 
 
\medskip 
 
\noindent (2) We remark that in higher dimensions, with a similar analysis, it can be shown  that $\tsr(\calE'(\mR^d))\leq d+1$. 

\medskip 
 
\noindent (3) The Bass stable rank (a notion from algebraic $K$-theory, recalled below) 
of $\calE'(\mR)$ is not known. 

If $\calA$ is a commutative unital algebra, then 
 $(a_1,\dots,a_n,b)\in  U_{n+1}(\calA)$ is  called {\it reducible} if there exists an $n$-tuple 
 $(\alpha_1,\cdots,\alpha_n)\in \calA^n$ such that we have 
  $
 (a_1+\alpha_1 b,\cdots, a_n+\alpha_n b)\in U_n(\calA).
 $ 
 It can be seen that if every element of $U_{n+1}(\calA)$ is reducible, then every element of  
 $U_{n+2}(\calA)$ is reducible too. The {\em Bass stable rank} of $\calA$, denoted by $\textrm{bsr} \;\!\calA$, 
 is the smallest $n\in \mN$ such that every element in $U_{n+1}(\calA)$ is reducible, and if 
  no such $n$ exists, then $\textrm{bsr} \;\!\calA:=\infty$. It is known that for commutative unital Banach algebras $\calA$, 
 $\textrm{bsr}\;\! \calA\leq \tsr \;\!\calA$ \cite[Theorem~3]{CorLar}. But the validity of such an inequality in the 
 context of topological algebras does not seem to be known. 
 We conjecture that  $\textrm{bsr}(\calE'(\mR))=2$.

 \goodbreak
 
\noindent (4) 
There 
are also several other natural convolution algebras of distributions on $\mR$, for example
\begin{eqnarray*}
 \calD_{\scriptscriptstyle{\geq \bigcdot}}'(\mR)&:=&\{T\in \calD'(\mR): \textrm{supp}(T) \textrm{ is bounded on the left}\}, \\
 \calD_{\scriptscriptstyle{\geq 0}}'(\mR)&:=&\{T\in \calD'(\mR): \textrm{supp}(T)\subset[0,\infty)\},
\end{eqnarray*}
 and we leave the determination of the stable ranks  
of these algebras as open questions. 

\medskip 
 
\noindent (5) \cite[Corollary~3.1]{MaaSas} gives a `corona-type' pointwise condition for coprimeness in $\calE'(\mR)$, reminiscent of the famous Carleson corona condition\footnote{The Hardy algebra $H^\infty(\mD)$ is the Banach algebra of all bounded and holomorphic functions on the unit disk $\mD:=\{z\in \mC:|z|<1\}$. The Carleson Corona Theorem \cite{Car} says that $(f_1,f_2)\in U_2(H^\infty(\mD))$ if and only if there exists a $\delta>0$ such that for all $z\in \mD$, $|f_1(z)|+|f_2(z)|>\delta$.} of coprimeness in the Banach algebra $H^\infty(\mD)$:

 \smallskip 
 
 \noindent $T_1, T_2\in U_2(\calE'(\mR))$ if and only if 
 there exist positive $C,N,M$ such that 
 $
 \textrm{ for all }z\in \mC,\;\;
 |\widehat{T_1}(z)|+|\widehat{T_2}(z)|\geq C(1+|z|^2)^{-N} e^{-M|\textrm{Im}(z)|}.
 $
\end{remarks}

\medskip 

\noindent {\bf Acknowledgements:} I thank Professor Michael Kunzinger
(University of Vienna) for answering my query on whether density in
$\calE'$ implies sequential density, for the reference \cite{Obe}, and for 
 useful comments. I also thank Professor Peter Wagner (University of Innsbruck) 
 for showing me the alternative approach to establishing Proposition~\ref{prop_geq}, 
 which we have included in this article (Proposition~\ref{prop_PW}).

\end{document}